\documentclass[11pt]{article}
\usepackage{bbm}
\usepackage{eqnarray,amsmath,amsfonts,amsthm,mathrsfs,url}
\usepackage{algorithmic}
\usepackage{algorithm}
\usepackage{float}
\usepackage{color}
\usepackage{bm}
\usepackage{amssymb}

\usepackage{fancyhdr}
\usepackage[top=3cm, bottom=3cm, left=3cm, right=3cm]{geometry}
\newtheorem{theorem}{Theorem}
\newtheorem{assumption}{Assumption}
\newtheorem{corollary}[theorem]{Corollary}
\newtheorem{definition}{Definition}

\newtheorem{lemma}{Lemma}[section]
\newtheorem{proposition}{Proposition}[section]
\newtheorem{remark}{Remark}[section]
\newtheorem{example}{Example}[section]

\allowdisplaybreaks[4]

\def\begeqn{\begin{equation}}
\def\endeqn{\end{equation}}
\def\begth{\begin{theorem}}
\def\endth{\end{theorem}}
\def\begprop{\begin{proposition}}
\def\endprop{\end{proposition}}
\def\begcor{\begin{corollary}}
\def\endcor{\end{corollary}}
\def\begdef{\begin{definition}}
\def\enddef{\end{definition}}
\def\beglemm{\begin{lemma}}
\def\endlemm{\end{lemma}}
\def\begexm{\begin{example}}
\def\endexm{\end{example}}
\def\begrem{\begin{remark}}
\def\endrem{\end{remark}}
\def\begassum{\begin{assumption}}
\def\endassum{\end{assumption}}

\begin{document}

\title{Faster Convergence of a Randomized Coordinate Descent Method for Linearly Constrained Optimization Problems}

\author{Qin Fang$^\dag$, Min Xu$^\ddag$ and Yiming Ying$^\S$\\
\\ \\
$^\dag$Information and Engineering College\\ Dalian University, Dalian 116622, China\\
$^\ddag$ School of Mathematical Sciences \\
Dalian University of Technology, Dalian 116024, China\\
$^\S$Department of Mathematics and Statistics\\
State University of New York, Albany, NY 12222, USA}

\date{}

\maketitle

\begin{abstract}
The problem of minimizing a separable convex function under linearly coupled constraints arises from various application domains such as economic systems, distributed control, and network flow. The main challenge for solving this problem is that the size of data is very large, which makes usual gradient-based methods infeasible. Recently, Necoara, Nesterov, and Glineur \cite{NecoaraNG15} proposed an efficient randomized coordinate descent method  to solve this type of optimization problems and presented an appealing convergence analysis. In this paper, we develop new techniques to analyze the convergence of such algorithms, which are able to greatly improve the results presented in \cite{NecoaraNG15}. This refined result is achieved by extending Nesterov's second technique \cite{Nesterov12} to the general  optimization problems with linearly coupled constraints.   A novel technique in our analysis is to establish the basis vectors for the subspace of the linearly constraints. 
\end{abstract}

\parindent=0cm

\section{Introduction}

Randomized block-coordinate descent (RBCD) method randomly chose one block of coordinates for updating according to a prescribed probability distribution at each iteration.  It has been proven to be very efficient to solve large-scale optimization problems \cite{LuX15,Nesterov12,RichtarikT14,Shalev-ShwartzT11,Shalev-ShwartzZ16}. In particular, Nesterov \cite{Nesterov12} studied RBCD methods for solving smooth convex problems and derived non-asymptotic convergence rates  in expectation without strong convexity or uniqueness assumptions. Specifically, according to the discussion in \cite{LuX15} there are roughly two techniques developed in this seminal paper for analyzing the convergence of the RBCD methods. The first technique is developed for the case where the blocks are randomly drawn from certain non-uniform  distribution. The second technique can yield better convergence rates which, however, only work for the uniform distribution. 

 Richt\'{a}rik and Tak\'{a}\v{c} \cite{RichtarikT14} established a high-probability type of iteration complexity of RBCD methods for composite convex problems by adopting Nesterov's first technique mentioned above.      
Lu and Xiao \cite{LuX15} obtained sharper expected-value type
of convergence rates of RBCD methods for composite convex problems  than those in  \cite{RichtarikT14}.  It can be regarded as an   extension and refinement of  Nesterov's second technique in \cite{Nesterov12} to composite optimization problems.  Parallel implementations of block-coordinate descent methods are studied in \cite{Necoara2013,Richtarik2016}.  In addition,  randomized block-coordinate descent methods under less conservative conditions have been also proposed in \cite{Necoara2016}.

In this paper, we are concerned with the following problem of minimizing a separable convex function subject to linearly coupled constraints:
\begin{equation}\label{eq:problem}
\begin{split}
&\min_{x_1,\ldots,x_N\in\mathbb{R}^{n}}\Big\{f_1(x_1)+\cdots+f_N(x_N)\Big\}\\
&~~~~~\text{s.t.:}~~~~~~x_1+\cdots+x_N=0.
\end{split}
\end{equation}
This simple model can be considered as the optimal resource allocation problem over a network. More precisely, one can interpret it as $N$ agents exchanging $n$ goods to minimize the total cost, where the constraints $\sum_{i=1}^Nx_i = 0$ are the market clearing or equilibrium requirements. Problem (\ref{eq:problem}) has many real applications in economic systems \cite{XiaoB06}, distributed computer systems \cite{KuroseS89}, distributed control \cite{NecoaraND11}, and network flow \cite{Bertsekas99}. In these applications, the size of data involved in formulation \eqref{eq:problem}  (i.e., $N$) is so large that usual methods based on full gradient computations are prohibitive.

Most of the aforementioned studies on RBCD methods  are concerned with optimization problems with decoupled constraints that can be characterized as a Cartesian product of certain convex sets. Hence, they can not be directly applied to the optimization problem \eqref{eq:problem} with coupled constraints, despite its extreme importance in practice.  Recently, Necoara, Nesterov, and Glineur \cite{NecoaraNG15} proposed an efficient randomized coordinate descent method for linearly constrained optimization over networks  and presented a convergence analysis in expectation.  The authors proposed the randomized $\tau$-block ($\tau\ge 2$) coordinate descent method for solving problem \eqref{eq:problem}, in which  $\tau$ blocks are randomly chosen for updating at each iteration.  The algorithm in \cite{NecoaraNG15} has been extended recently to composite problems \cite{Necoara2014} and  more general optimization models \cite{Necoara2013b}.   It can be regarded as an extension of Nesterov's first technique for smooth  optimization problems with decoupled constraints to the case of  linearly coupled constraints.  

As mentioned above,   Nesterov's second technique  can yield better rates as shown in \cite{LuX15,Nesterov12} for optimization problems with {\em  decoupled constraints}.  Moreover, it has been demonstrated so far that it works only for the RBCD methods  where the random blocks are  drawn from the {\em uniform} distribution. A natural question is that whether Nesterov's second technique works for smooth optimization problem \eqref{eq:problem} with linearly coupled constraints and    
certain non-uniform distributions. 

In this paper, we successfully give an affirmative answer to  the above question and establish significantly sharper convergence rates in expectation than that presented in \cite{NecoaraNG15}.    As a result, better iteration complexities in probability are also derived. The spirit of our proofs is similar to  those in various papers \cite{Nesterov12,NecoaraNG15} and particularly \cite{LuX15}. 
The main novelty in our analysis is to establish the basis vectors for the subspace of the linearly coupled  constraints in problem \eqref{eq:problem}.  These new technical lemmas are stated as Lemmas \ref{lem:imp1} and \ref{lem:imp2} in the subsequent sections. These are the key ingredients to obtain tighter results.

This paper is organized as follows. In Section 2 we review  the randomized $2$-block coordinate descent method for problem \eqref{eq:problem} proposed in \cite{NecoaraNG15} and present our main results. In Section 3 we first present some technical lemmas, and then provide the proofs of the main results. Finally, some concluding remarks are given in Section 4.

\section{Main results}\label{sec:main}
We begin with the introduction of randomized $\tau$-block coordinate descent method for problem \eqref{eq:problem} proposed in \cite{NecoaraNG15}. For simplicity, we only focus on the case of $\tau=2$.   To this end, let us introduce some notations.

Let $\mathbb{R}^{nN}$ denote the whole coordinate space, and $I_{nN}=[U_1, \ldots, U_N]$ a decomposition of the identity matrix $I_{nN}$ into $N$ submatrices $U_i\in \mathbb{R}^{nN\times n}$. Then, one can write any vector $x\in\mathbb{R}^{nN}$ in terms of its blocks $x_i\in \mathbb{R}^{n}$ as $x=\sum_{i=1}^NU_ix_i$. Meanwhile, each block $x_i\in \mathbb{R}^{n}$ of $x\in \mathbb{R}^{nN}$ can be represented uniquely as $x_i=U_i^Tx$. Denote by $\langle\cdot,\cdot\rangle$ and $\Vert\cdot\Vert$ the standard Euclidean inner product and norm on the subspace $\mathbb{R}^{n}$, respectively.

Throughout this paper, we assume the following basic assumption.
\begin{assumption}\label{assum:basic}
Each function $f_i$, $i=1,\dots,N$, is convex and has Lipschitz continuous gradient with Lipschitz constants $L_i>0$, i.e.,
\begin{equation}\label{eq:blocklip}
\Vert\nabla f_i(x_i+d_i)-\nabla f_i(x_i)\Vert\leq L_i \Vert d_i\Vert,~~~\forall x_i,d_i\in \mathbb{R}^{n}.
\end{equation}
\end{assumption}

Let $E$ denote the set of all pairs of indices $(i,j)$ with $1\leq i<j\leq N$. By using Lipschitz constants $L_i$, define a set of positive numbers $p_{ij}$ by
\begin{equation}\label{eq:probability}
p_{ij}:=\frac{\frac{1}{L_i}+\frac{1}{L_j}}{(N-1)\sum_{t=1}^N\frac{1}{L_t}},~~~(i,j)\in E.
\end{equation}
It is easy to check that $\sum_{(i,j)\in E}p_{ij}=1$, which means that $\{p_{ij}\}_{(i,j)\in E}$ can be taken as a probability vector over $E$. The above specific choice of $p_{ij}$ has been suggested in \cite{NecoaraNG15}.

Let $S=\left\{x\in\mathbb{R}^{nN}:\sum_{i=1}^{N}x_i=0\right\}$. For the current iterate $x\in S$ and a pair of indexes $(i,j)\in E$, to maintain the feasibility of the next iterate $x^{+}$, we define it in the following way:
$$x^{+}:=x+U_id_i(x)+U_jd_j(x),$$
where
\begin{eqnarray}
\left[d_i^T(x),d_j^T(x)\right]^T&=&\arg\min_{d_i+d_j=0}\left\{\langle\nabla f_i(x_i),d_i\rangle+\langle\nabla f_j(x_j),d_j\rangle+\frac{L_i}{2}\Vert d_i\Vert^2+\frac{L_j}{2}\Vert d_j\Vert^2\right\}\nonumber\\
&=&-\frac{1}{L_i+L_j}\left[(\nabla f_i(x_i)-\nabla f_j(x_j))^T,(-\nabla f_i(x_i)+\nabla f_j(x_j))^T\right]^T.\label{eq:optdir}
\end{eqnarray}
Let $x^0\in S$ be the initial point. Then the randomized 2-block coordinate descent method  proposed in \cite{NecoaraNG15}  can be described as follows:
\begin{algorithm}[H]
\caption{Randomized 2-Block Coordinate Descent Method}
\label{alg:R2BCDM}
\begin{algorithmic}
\STATE{Repeat for $k=0,1,2,\ldots$}
\STATE{1. Choose a pair of indices $(i_k,j_k)$ from $E$ randomly with the probability $p_{i_k,j_k}$.}
\STATE{2. Update $x^{k+1}=x^k+U_{i_k}d_{i_k}(x^k)+U_{j_k}d_{j_k}(x^k)$.}
\end{algorithmic}
\end{algorithm}

For $x\in \mathbb{R}^{nN}$, define $f(x):=f_1(x_1)+\cdots+f_N(x_N)$.
Let $X^{\ast}$ denote the set of optimal solutions to problem (\ref{eq:problem}), and $f^{\ast}$ the corresponding optimal value.
Equip the space $\mathbb{R}^{nN}$ with the norm
$\Vert x\Vert_L=\left(\sum_{i=1}^NL_i\Vert x_i\Vert^2\right)^{\frac{1}{2}}$, and
define $\tilde{R}(x^0):=\min_{x^\ast\in X^\ast}\left\{\Vert x^0-x^\ast\Vert_L\right\}$.
  From (\ref{eq:blocklip}), we know that
\begin{equation}\label{eq:bloquaapp}
f_i(x_i+d_i)\leq f_i(x_i)+\langle\nabla f_i(x_i),d_i\rangle+\frac{L_i}{2}\Vert d_i\Vert^2,~~~\forall x_i,d_i\in \mathbb{R}^{n}.
\end{equation}
Summing the inequalities (\ref{eq:bloquaapp}) for $i=1,\ldots,N$, we get  
\begin{equation}\label{eq:gloquaapp}
f(x+d)\leq f(x)+\langle\nabla f(x),d\rangle+\frac{1}{2}\Vert d\Vert_L^2,~~~\forall x,d\in \mathbb{R}^{nN}.
\end{equation}
Let $\mu_f$ denote the convexity parameter of $f$ with respect to the norm $\Vert \cdot\Vert_L$, i.e., the greatest $\mu\geq 0$ such that
\begin{equation}\label{eq:y1}
f(y)\geq f(x)+\langle\nabla f(x),y-x\rangle+\frac{\mu}{2}\Vert y-x\Vert_L^2,~~~\forall x,y\in \mathbb{R}^{nN}.
\end{equation}

Let $\eta_k$ denote the random variable
$\left\{(i_0,j_0),(i_1,j_1),\ldots,(i_{k-1},j_{k-1})\right\}$.

With these notations, we now state one of our main results.

\begin{theorem}\label{thm:coninexp}
Let $\{x^k\}$ be the sequence generated by Algorithm \ref{alg:R2BCDM}. Then for any $k\geq 0$,
\begin{equation}\label{eq:sublinear}
\mathbf{E}_{\eta_{k}}\left[f(x^{k})\right]-f^\ast
\leq\frac{N-1}{N+k-1}\tilde{R}^2(x^0).
\end{equation}
In addition, if $f$ is strongly convex with parameter $\mu_f>0$, then
\begin{equation}\label{eq:linear}
\mathbf{E}_{\eta_{k}}\left[f(x^{k})\right]-f^\ast
\leq\left(1-\frac{2\mu_f}{(N-1)(1+\mu_f)}\right)^k\tilde{R}^2
(x^0).
\end{equation}
\end{theorem}

As a result of Theorem \ref{thm:coninexp}, we may derive the following high-probability type of iteration complexities of Algorithm \ref{alg:R2BCDM} for finding an $\epsilon$-optimal solution of problem (\ref{eq:problem}).
\begin{theorem}\label{thm:coninpro}
Let $\epsilon\in(0,f(x^0)-f^\ast)$ and $\rho\in(0,1)$ be the desired accuracy and the confidence level, respectively. Let $\{x^k\}$ be the sequence generated by Algorithm \ref{alg:R2BCDM}. Then for all $k\geq K$, there holds
\begin{equation}\label{eq:convergencprob}
\mathbf{Prob}\left[f(x^k)-f^\ast\leq\epsilon\right]\geq 1-\rho,
\end{equation}
where
\begin{equation}\label{eq:iterationconvexour}
K:=\frac{2(N-1)R^2(x^0)}{\epsilon}\left[1+\log\left(\frac{
\tilde{R}^2(x^0)}{2R^2(x^0)\rho}\right)\right]-N+3.
\end{equation}
In addition, if $f$ is strongly convex with parameter $\mu_f>0$, then
(\ref{eq:convergencprob}) holds when $k\geq \tilde{K}$, where
\begin{equation}\label{eq:iterationstronglyour}
\tilde{K}:=\frac{(N-1)(1+\mu_f)}{2\mu_f}\log\left(\frac{\tilde{R}^2(x^0)}{\rho\epsilon}\right).
\end{equation}
\end{theorem}

The proofs for the above theorems will be given in Section \ref{sec:proof}. Let us compare our results with  those in \cite{NecoaraNG15}.

Firstly, for the case when $f$ is a smooth convex function, Necoara, Nesterov, and Glineur \cite[Theorem 3.1]{NecoaraNG15} showed that
\begin{equation}\label{eq:Necoarasublin}
\mathbf{E}_{\eta_{k}}\left[f(x^{k})\right]-f^\ast\leq\frac{2(N-1)R^2(x^0)}{k},
\end{equation}
where $R(x^0):=\max_{x\in S}\min_{x^\ast\in X^\ast}\left\{\Vert x-x^\ast\Vert_L:f(x)\leq f(x^0)\right\}$. Let $A$ and $B$ denote the right-hand side of (\ref{eq:sublinear}) and (\ref{eq:Necoarasublin}), respectively. Then we have
\begin{equation*}
\frac{B}{A}=2\left(\frac{N-1}{k}+1\right)\frac{R^2(x^0)}{\tilde{R}^2(x^0)}\geq 2\left(\frac{N-1}{k}+1\right).
\end{equation*}
Since in practice, the number of blocks $N$ in problem (\ref{eq:problem}) is usually of large scale and $\tilde{R}(x^0)$ could be much smaller than $R(x^0)$, we thus conclude that our bound \eqref{eq:sublinear} is significantly smaller than \eqref{eq:Necoarasublin}.

Secondly, for the case when $f$ is a smooth strongly convex function, Necoara, Nesterov, and Glineur \cite[Theorem 3.2]{NecoaraNG15} showed that
\begin{equation}\label{eq:Necoaralin}
\mathbf{E}_{\eta_{k}}\left[f(x^{k})\right]-f^\ast\leq\left(1-\frac{\mu_f}{N-1}\right)^k(f(x^0)-f^\ast).
\end{equation}
From the inequalities \eqref{eq:gloquaapp} and \eqref{eq:y1}, we arrive at the fact that $\mu_f\leq 1.$ 
This implies that
$1-\frac{2\mu_f}{(N-1)(1+\mu_f)}\leq 1-\frac{\mu_f}{N-1}$.
Thus, the convergence rate (\ref{eq:linear}) is much sharper than the rate (\ref{eq:Necoaralin}).

Using the same argument as in the proof of Theorem \ref{thm:coninpro}, the convergence with high probability  can be straightforwardly derived from \cite[Theorems 3.1 and 3.2]{NecoaraNG15} as follows: If $f$ is a smooth convex function, then (\ref{eq:convergencprob}) holds for all $k\geq \bar{K}$, where
\begin{equation}\label{eq:iterationconvexNNG}
\bar{K}:=\frac{2(N-1)R^2(x^0)}{\epsilon}\left(1+\log\frac{1}
{\rho}\right)+2.
\end{equation}
If, in addition, $f$ is strongly convex with parameter $\mu_f>0$, then (\ref{eq:convergencprob}) holds for all $k\geq \hat{K}$, where
\begin{equation}\label{eq:iterationstrongNNG}
\hat{K}:=\frac{N-1}{\mu_f}\log\left(\frac{f(x^0)-f^\ast}{\epsilon\rho}\right).
\end{equation}

Let us make a comparison between the complexities $K$ and $\bar{K}$. It follows from (\ref{eq:iterationconvexour}) and (\ref{eq:iterationconvexNNG}) that
\begin{eqnarray*}
\bar{K}-K&=&\frac{2(N-1)R^2(x^0)}{\epsilon}\log\left(\frac{2R^2(x^0)}{\tilde{R}^2(x^0)}\right)+N-1\\
         &\geq&\frac{(N-1)R^2(x^0)\log4}{\epsilon}+N-1.
\end{eqnarray*}
Thus, the iteration complexity $\bar{K}$ implied in \cite{NecoaraNG15} is substantially larger than ours. To compare $\tilde{K}$ with $\hat{K}$, we have by (\ref{eq:iterationstronglyour}) and (\ref{eq:iterationstrongNNG}) that for sufficiently small $\rho$ or $\epsilon$,
\begin{equation*}
\frac{\hat{K}}{\tilde{K}}\approx\frac{2}{1+\mu_f}\geq 1
\end{equation*}
due to the fact that $\mu_f\leq 1$. Therefore, our iteration complexity $\tilde{K}$ is tighter than $\hat{K}$ when $\rho$ or $\epsilon$ is sufficiently small.

\section{Proofs of main results}\label{sec:proof}
In this section we provide the proofs of Theorems \ref{thm:coninexp} and \ref{thm:coninpro}. 
For this purpose, we need first establish some preliminary results including three technical lemmas.

\subsection{Preliminary results}\label{sunsec:preliminaries}

Let $\{e_l:1\leq l\leq N\}$ and $\{\tilde{e}_m:1\leq m\leq n\}$ denote standard basis vectors in $\mathbb{R}^{N}$ and $\mathbb{R}^{n}$, respectively. Define, for $l=1,\ldots,N-1$ and $m=1,\ldots,n$, $v_{lm}:=(e_l-e_{l+1})\otimes\tilde{e}_m$, where $\otimes$ denotes the Kronecker product. Then we have
\begin{lemma}\label{lem:imp1}
The vectors $\{v_{lm}:1\leq l\leq N-1,1\leq m\leq n\}$ form a basis for the space $S$.
\end{lemma}
\begin{proof}
Let $e\in\mathbb{R}^{N}$ denote the column vector with all entries 1 and let $I_n$ denote the identity matrix of order $n$. Then we have
\begin{equation*}
(e^T\otimes I_n)v_{lm}=(e^T\otimes I_n)((e_l-e_{l+1})\otimes\tilde{e}_m)=(e^T(e_l-e_{l+1}))\otimes\tilde{e}_m=0,
\end{equation*}
which immediately implies that $v_{lm}\in S$. On the other hand, note that any $x\in S$ has the following representation:
\begin{equation*}
x=[x_1^T,\ldots,x_{N-1}^T, -x_1^T-\cdots-x_{N-1}^T]^T=\sum_{l=1}^{N-1}\sum_{m=1}^n \langle\tilde{e}_m,x_1+x_2+\cdots+x_l\rangle v_{lm}.
\end{equation*}
Thus, the vectors $\{v_{lm}:1\leq l\leq N-1,1\leq m\leq n\}$ form a basis for the space $S$.
\end{proof}

With the help of Lemma \ref{lem:imp1}, we can prove the following lemma which plays a critical role in our proof of the main results.
\begin{lemma}\label{lem:imp2}
Let $\{p_{ij}\}_{(i,j)\in E}$ be the probability vector over $E$ defined by (\ref{eq:probability}). Then, for all $x\in S$, we have
\begin{equation}\label{eq:imp}
\sum_{(i,j)\in E}\left[\frac{p_{ij}}{L_i+L_j}(e_i-e_j)(L_ie_i^T-L_je_j^T)\otimes I_n\right] x =\frac{1}{N-1}x.
\end{equation}
\end{lemma}
\begin{proof}
By Lemma \ref{lem:imp1}, it suffices to show that (\ref{eq:imp}) holds for all $v_{lm}$.
Let $\delta_{ij}$ denote the Kronecker delta: $\delta_{ij}=1$ if $i=j$, and $\delta_{ij}=0$ otherwise. By (\ref{eq:probability}), we have
\begin{eqnarray*}
&&(N-1)\left(\sum_{t=1}^N\frac{1}{L_t}\right)\sum_{(i,j)\in E}\left[\frac{p_{ij}}{L_i+L_j}(e_i-e_j)(L_ie_i^T-L_je_j^T)(e_l-e_{l+1}) \right]\\
&\overset{(\ref{eq:probability})}{=}&\sum_{(i,j)\in E}\left[\frac{1}{L_iL_j}(e_i-e_j)(L_i\delta_{il}-L_i\delta_{i,l+1}-L_j\delta_{jl}+L_j\delta_{j,l+1}) \right]\\
&=&\sum_{(i,j)\in E}\left[\frac{1}{L_j}(e_i-e_j)(\delta_{il}-\delta_{i,l+1})\right]+\sum_{(i,j)\in E}\left[\frac{1}{L_i}(e_i-e_j)(-\delta_{jl}+\delta_{j,l+1})\right]\\
&=&\sum_{j>l}\frac{1}{L_j}(e_l-e_j)-\sum_{j>l+1}\frac{1}{L_j}(e_{l+1}-e_j) -\sum_{i<l}\frac{1}{L_i}(e_i-e_l)+\sum_{i<l+1}\frac{1}{L_i}(e_i-e_{l+1})\\
&=&\frac{1}{L_{l+1}}(e_l-e_{l+1})+\sum_{j>l+1}\frac{1}{L_j}(e_{l}-e_{l+1}) +\sum_{i<l}\frac{1}{L_i}(e_l-e_{l+1}) +\frac{1}{L_l}(e_l-e_{l+1})\\
&=&\left(\sum_{i=1}^N \frac{1}{L_i}\right)(e_l-e_{l+1}),
\end{eqnarray*}
and hence it follows that
\begin{eqnarray*}
\sum_{(i,j)\in E}\left[\frac{p_{ij}}{L_i+L_j}(e_i-e_j)(L_ie_i^T-L_je_j^T)(e_l-e_{l+1}) \right]=\frac{1}{N-1}(e_l-e_{l+1}).
\end{eqnarray*}
In view of this and recalling that $v_{lm}=(e_l-e_{l+1})\otimes\tilde{e}_m$, we get
\begin{eqnarray*}
&&\sum_{(i,j)\in E}\left[\frac{p_{ij}}{L_i+L_j}(e_i-e_j)(L_ie_i^T-L_je_j^T)\otimes I_n\right] v_{lm}\\
&=&\sum_{(i,j)\in E}\left[\frac{p_{ij}}{L_i+L_j}(e_i-e_j)(L_ie_i^T-L_je_j^T)(e_l-e_{l+1})\right] \otimes\tilde{e}_m\\
&=&\frac{1}{N-1}(e_l-e_{l+1})\otimes\tilde{e}_m=\frac{1}{N-1}v_{lm},
\end{eqnarray*}
and thus (\ref{eq:imp}) holds for all $x\in S$.
\end{proof}

\begin{lemma}\label{lem:imp3}
Let $x^0\in S$ be the initial point of Algorithm \ref{alg:R2BCDM}, then we have
\begin{equation}\label{eq:keyineq}
f(x^0)-f^\ast\leq\frac{1}{2}\tilde{R}^2(x^0).
\end{equation}
\end{lemma}
\begin{proof}
Let $x^\ast$ be an arbitrary optimal solution of problem (\ref{eq:problem}).
Then it satisfies the following optimality condition
\begin{equation}\label{eq:optcondi}
x^\ast\in S~~~\text{and}~~~\nabla f(x^\ast)\in S^{\perp},
\end{equation}
where $S^\perp=\Big\{x\in\mathbb{R}^{nN}:x_1=x_2=\cdots=x_N\Big\}$ denotes the orthogonal complement of $S$.
By using (\ref{eq:gloquaapp}) and (\ref{eq:optcondi}), we get
\begin{equation*}
f(x^0)-f^\ast\leq\langle\nabla f(x^\ast),x^0-x^\ast\rangle+\frac{1}{2}\Vert x^0-x^\ast\Vert_L^2=\frac{1}{2}\Vert x^0-x^\ast\Vert_L^2.
\end{equation*}
By the definition of $\tilde{R}(x^0)$, the desired result follows.
\end{proof}

\subsection{Proofs of Theorems \ref{thm:coninexp} and \ref{thm:coninpro}}\label{subsec:proof}
\begin{proof}[Proof of Theorem \ref{thm:coninexp}]

For $k\geq 0$, denote
\begin{equation*}
r_k^2=\Vert x^k-x^\ast\Vert_L^2=\sum_{i=1}^NL_i\left\Vert x_i^k-x_i^\ast\right\Vert^2.
\end{equation*}
Similar to the proof in \cite{LuX15}, we are going to derive the relationship between $\mathbf{E}\left[\frac{1}{2}r_{k+1}^2+f(x^{k+1})-f^\ast\right]$ and $\mathbf{E}\left[\frac{1}{2}r_k^2+f(x^k)-f^\ast\right].$
To this end, noting that $x^{k+1}=x^k+U_{i_k}d_{i_k}(x^k)+U_{j_k}d_{j_k}(x^k)$ and in view of (\ref{eq:optdir}), we obtain
\begin{eqnarray*}
r_{k+1}^2&=&\left\Vert x^k-x^\ast+U_{i_k}d_{i_k}(x^k)+U_{j_k}d_{j_k}(x^k)\right\Vert_L^2\\
&=&\sum_{i\notin\{i_k,j_k\}}L_i\left\Vert x_i^k-x_i^\ast\right\Vert^2+\sum_{i\in\{i_k,j_k\}}L_i\left\Vert x_i^k-x_i^\ast+d_{i}(x^k)\right\Vert^2\\
&=&r_k^2+\sum_{i\in\{i_k,j_k\}}L_i\left\Vert d_{i}(x^k)\right\Vert^2+2\sum_{i\in\{i_k,j_k\}}L_i\langle x_i^k-x_i^\ast,d_{i}(x^k)\rangle\\
&\overset{(\ref{eq:optdir})}{=}&r_k^2+\frac{1}{L_{i_k}+L_{j_k}}\left\Vert \nabla f_{i_k}(x_{i_k}^k)-\nabla f_{j_k}(x_{j_k}^k) \right\Vert^2+2\sum_{i\in\{i_k,j_k\}}L_i\langle x_i^k-x_i^\ast,d_{i}(x^k)\rangle\\
&=&r_k^2+\frac{1}{L_{i_k}+L_{j_k}}\left\Vert \nabla f_{i_k}(x_{i_k}^k)-\nabla f_{j_k}(x_{j_k}^k)\right\Vert^2+\frac{2}{L_{i_k}+L_{j_k}}\bigg[\left\langle L_{i_k}U_{i_k}^T(x^\ast-x^k),\right.\\
&&\left.(U_{i_k}-U_{j_k})^T\nabla f(x^k)\right\rangle-\left\langle L_{j_k}U_{j_k}^T(x^\ast-x^k),(U_{i_k}-U_{j_k})^T\nabla f(x^k)\right\rangle\bigg].
\end{eqnarray*}
Recalling the definition of $f$, from (\ref{eq:blocklip})  we further obtain that for all $x_i,d_i\in \mathbb{R}^{n}$,
\begin{equation*}
f(x+U_id_i+U_jd_j)\leq f(x)+\langle\nabla f_i(x_i),d_i\rangle+\langle\nabla f_j(x_j),d_j\rangle+\frac{L_i}{2}\Vert d_i\Vert^2+\frac{L_j}{2}\Vert d_j\Vert^2,
\end{equation*}
which together with (\ref{eq:optdir}) implies that
\begin{equation}\label{eq:deslem}
f(x)-f(x+U_id_i(x)+U_jd_j(x))\geq \frac{1}{2(L_i+L_j)}\Vert\nabla f_i(x_i)-\nabla f_j(x_j)\Vert^2.
\end{equation}
From  (\ref{eq:deslem}), we further obtain that
\begin{eqnarray*}
r_{k+1}^2
&=&r_k^2+\frac{1}{L_{i_k}+L_{j_k}}\left\Vert\nabla f_{i_k}(x_{i_k}^k)-\nabla f_{j_k}(x_{j_k}^k)\right\Vert^2+\frac{2}{L_{i_k}+L_{j_k}}\bigg[\left\langle (U_{i_k}-U_{j_k})(L_{i_k}U_{i_k}^T\right.\\
&&\left.-L_{j_k}U_{j_k}^T)(x^\ast-x^k),\nabla f(x^k)\right\rangle\bigg]\\
&=&r_k^2+\frac{1}{L_{i_k}+L_{j_k}}\left\Vert\nabla f_{i_k}(x_{i_k}^k)-\nabla f_{j_k}(x_{j_k}^k)\right\Vert^2+\frac{2}{L_{i_k}+L_{j_k}}\bigg[\left\langle (e_{i_k}-e_{j_k})(L_{i_k}e_{i_k}^T\right.\\
&&\left.-L_{j_k}e_{j_k}^T)\otimes I_n(x^\ast-x^k),\nabla f(x^k)\right\rangle\bigg]\\
&\overset{(\ref{eq:deslem})}{\leq}&r_k^2+2(f(x^k)-f(x^{k+1}))+\frac{2}{L_{i_k}+L_{j_k}}\bigg[\left\langle (e_{i_k}-e_{j_k})(L_{i_k}e_{i_k}^T-L_{j_k}e_{j_k}^T)\otimes I_n(x^\ast\right.\\
&&\left.-x^k),\nabla f(x^k)\right\rangle\bigg].
\end{eqnarray*}
First multiplying both sides of the above inequality by $1/2$, then taking the expectation with respect to $(i_k,j_k)$, and finally rearranging terms, we arrive at
\begin{eqnarray*}
\mathbf{E}_{(i_k,j_k)}\left[\frac{1}{2}r_{k+1}^2+f(x^{k+1})-f^\ast\right]
&\leq&\left[\frac{1}{2}r_k^2+f(x^k)-f^\ast\right]+\Bigg\langle \sum_{(i,j)\in E}\bigg[\frac{p_{ij}}{L_i+L_j}(e_i\\
&&-e_j)(L_ie_i^T-L_je_j^T)\otimes I_n\bigg](x^\ast-x^k),\nabla f(x^k)\Bigg\rangle.
\end{eqnarray*}
Recall that $\eta_k=\left\{(i_0,j_0),(i_1,j_1),\ldots,(i_{k-1},j_{k-1})\right\}$. 
In light of Lemma \ref{lem:imp2}, it follows that
\begin{equation}\label{eq:iterel}
\mathbf{E}_{(i_k,j_k)}\left[\frac{1}{2}r_{k+1}^2+f(x^{k+1})-f^\ast\right]
\leq\left[\frac{1}{2}r_k^2+f(x^k)-f^\ast\right]+\frac{1}{N-1}\left\langle x^\ast-x^k,\nabla f(x^k)\right\rangle.
\end{equation}
Taking the expectation with respect to $\eta_{k}$ on both sides of (\ref{eq:iterel}) yields
\begin{eqnarray*}
\mathbf{E}_{\eta_{k+1}}\left[\frac{1}{2}r_{k+1}^2+f(x^{k+1})-f^\ast\right]
&\leq&\mathbf{E}_{\eta_{k}}\left[\frac{1}{2}r_k^2+f(x^k)-f^\ast\right]+
\frac{1}{N-1}\times\\
&&\mathbf{E}_{\eta_{k}}
\left[\left\langle x^\ast-x^k,\nabla f(x^k)\right\rangle\right].
\end{eqnarray*}
Notice that, by (\ref{eq:deslem}), the sequence $\{\mathbf{E}_{\eta_{k}}[f(x^k)]\}$
is decreasing. Using this fact and applying the above inequality recursively, we get
\begin{eqnarray*}
\mathbf{E}_{\eta_{k+1}}\left[f(x^{k+1})\right]-f^\ast&\leq&
\mathbf{E}_{\eta_{k+1}}\left[\frac{1}{2}
r_{k+1}^2+f(x^{k+1})-f^\ast\right]\\
&\leq&\frac{1}{2}r_0^2+f(x^0)-f^\ast+\frac{1}{N-1}\sum_{l=0}^k\mathbf{E}_{\eta_{l}}
\left[\left\langle x^\ast-x^l,\nabla f(x^l)\right\rangle\right]\\
&\leq&\frac{1}{2}r_0^2+f(x^0)-f^\ast-\frac{1}{N-1}\sum_{l=0}^k\left(\mathbf{E}_{\eta_{l}}
\left[f(x^{l})\right]-f^\ast\right)\\
&\leq&\frac{1}{2}r_0^2+f(x^0)-f^\ast-\frac{k+1}{N-1}\Big(\mathbf{E}_{\eta_{k+1}}
\left[f(x^{k+1})\right]-f^\ast\Big).
\end{eqnarray*}
Combining the last inequality with Lemma \ref{lem:imp3}, we obtain
\begin{equation*}
\mathbf{E}_{\eta_{k+1}}\left[f(x^{k+1})\right]-f^\ast
\leq\frac{N-1}{N+k}\left[\frac{1}{2}r_0^2+f(x^0)-f^\ast\right]\overset{(\ref{eq:keyineq})}{\leq}\frac{N-1}{N+k}\left[\frac{1}{2}
r_0^2+\frac{1}{2}\tilde{R}^2(x^0)\right],
\end{equation*}
which together with the arbitrariness of $x^\ast$ yields (\ref{eq:sublinear}).

Next, we shall prove (\ref{eq:linear}). By the strong convexity of $f$ and the optimality condition (\ref{eq:optcondi}), we obtain
\begin{eqnarray*}
\langle\nabla f(x^k),x^k-x^\ast\rangle&\geq& f(x^k)-f(x^\ast)+\frac{\mu_f}{2}\Vert x^\ast-x^k\Vert_L^2\\
&\geq&\langle\nabla f(x^\ast),x^k-x^\ast\rangle+\mu_f\Vert x^k-x^\ast\Vert_L^2\\
&\overset{(\ref{eq:optcondi})}{=}&\mu_f\Vert x^k-x^\ast\Vert_L^2.
\end{eqnarray*}
Let $\beta=\frac{2\mu_f}{1+\mu_f}\in[0,1]$. Then
\begin{eqnarray*}
\langle\nabla f(x^k),x^k-x^\ast\rangle&\geq& \beta \left[f(x^k)-f(x^\ast)+\frac{\mu_f}{2}\Vert x^\ast-x^k\Vert_L^2\right]+(1-\beta)
\mu_f\Vert x^\ast-x^k\Vert_L^2\\
&=&\frac{2\mu_f}{1+\mu_f}\left[\frac{1}{2}r_k^2+f(x^k)-f^\ast\right].
\end{eqnarray*}
Substituting this inequality into inequality (\ref{eq:iterel}) gives
\begin{equation*}
\mathbf{E}_{(i_k,j_k)}\left[\frac{1}{2}r_{k+1}^2+f(x^{k+1})-f^\ast\right]
\leq\left(1-\frac{2\mu_f}{(N-1)(1+\mu_f)}\right)\left[\frac{1}{2}r_k^2+f(x^k)-f^\ast\right].
\end{equation*}
Taking expectation with respect to $\eta_k$ on both sides of the above relation, we get
\begin{equation*}
\mathbf{E}_{\eta_{k+1}}\left[\frac{1}{2}r_{k+1}^2+f(x^{k+1})-f^\ast\right]
\leq\left(1-\frac{2\mu_f}{(N-1)(1+\mu_f)}\right)\mathbf{E}_{\eta_{k}}\left[\frac{1}{2}r_k^2+f(x^k)-f^\ast\right],
\end{equation*}
from which the inequality (\ref{eq:linear}) follows.
\end{proof}

\begin{proof}[Proof of Theorem \ref{thm:coninpro}]
Let $\Delta_k=f(x^k)-f^\ast$ for all $k\geq 0$. It was proved in \cite[Theorem 3.1]{NecoaraNG15} that
\begin{equation}\label{eq:recurel}
\mathbf{E}_{(i_k,j_k)}\left[\Delta_{k+1}\right]\leq \Delta_{k}-\frac{\Delta_{k}^2}{2(N-1)R^2(x^0)}.
\end{equation}
Define the sequence $\left\{\Delta_k^\epsilon\right\}$ as follows: $\Delta_k^\epsilon=\Delta_k$ if $\Delta_k\geq\epsilon$, and $\Delta_k^\epsilon=0$ otherwise.
Using (\ref{eq:recurel}) and the same argument as used in the proof of \cite[Theorem 1]{RichtarikT14}, we get
\begin{equation*}
\mathbf{E}_{(i_k,j_k)}\left[\Delta_{k+1}^\epsilon\right]\leq\left( 1-\frac{\epsilon}{2(N-1)R^2(x^0)}\right)\Delta_{k}^\epsilon.
\end{equation*}
Taking expectation with respect to $\eta_k$ on both sides of the above relation gives
\begin{equation}\label{eq:trunrecurel}
\mathbf{E}_{\eta_{k+1}}\left[\Delta_{k+1}^\epsilon\right]\leq\left( 1-\frac{\epsilon}{2(N-1)R^2(x^0)}\right)\mathbf{E}_{\eta_{k}}\left[\Delta_{k}^\epsilon\right].
\end{equation}
In addition, using (\ref{eq:sublinear}) and the relation $\Delta_{k}^\epsilon\leq\Delta_{k}$, we obtain
\begin{equation}\label{eq:trunsublin}
\mathbf{E}_{\eta_{k}}\left[\Delta_{k}^\epsilon\right]
\leq\frac{N-1}{N+k-1}\tilde{R}^2(x^0).
\end{equation}
Now for any $t>0$, let
$K_1=\left\lceil\frac{N-1}{t\epsilon}\tilde{R}^2(x^0)
\right\rceil-N+1$ and $K_2=\left\lceil\frac{2(N-1)R^2(x^0)}{\epsilon}\log\frac{t}
{\rho}\right\rceil$.
It follows from (\ref{eq:trunrecurel}) and (\ref{eq:trunsublin}) that
\begin{eqnarray*}
\mathbf{E}_{\eta_{K_1+K_2}}\left[\Delta_{K_1+K_2}^\epsilon\right]&\leq&\left( 1-\frac{\epsilon}{2(N-1)R^2(x^0)}\right)^{K_2}\mathbf{E}_{\eta_{K_1}}\left[\Delta_{K_1}^\epsilon\right]
\leq t\epsilon\left( 1-\frac{\epsilon}{2(N-1)R^2(x^0)}\right)^{K_2}\\
&\leq&t\epsilon\exp\left( -\frac{\epsilon K_2}{2(N-1)R^2(x^0)}\right)
\leq\rho\epsilon.
\end{eqnarray*}
Note that, by (\ref{eq:trunrecurel}), the sequence $\left\{\mathbf{E}_{\eta_{k}}\left[\Delta_{k}^\epsilon\right]\right\}$ is decreasing.
Therefore for $k\geq K(t)$, we have $\mathbf{E}_{\eta_{k}}\left[\Delta_{k}^\epsilon\right]\leq \rho\epsilon$,
where
\begin{equation*}
K(t)=\frac{(N-1)\tilde{R}^2(x^0)}{t\epsilon}
-N+3+\frac{2(N-1)R^2(x^0)}{\epsilon}\log\frac{t}
{\rho}.
\end{equation*}
By the definition of $K$, one can easily check that $K=K(t^\ast)$, where
$t^\ast:=\arg\min_{t>0}K(t)$. Finally, using the Markov inequality, we obtain that for any $k\geq K$
\begin{equation*}
\mathbf{Prob}\left(f(x^k)-f^\ast>\epsilon\right)=
\mathbf{Prob}\left(\Delta_k>\epsilon\right)
\leq\mathbf{Prob}\left(\Delta_k^\epsilon\geq\epsilon\right)\leq
\frac{\mathbf{E}_{\eta_{k}}\left[\Delta_{k}^\epsilon\right]}{\epsilon}
\leq\rho,
\end{equation*}
which immediately implies that the first claim holds.

We next show the second claim. Using the Markov inequality, the inequality (\ref{eq:linear}) and the definition of $\tilde{K}$ we
obtain that for any $k\geq \tilde{K}$
\begin{eqnarray*}
\mathbf{Prob}\left(f(x^k)-f^\ast>\epsilon\right)&\leq&\frac{\mathbf{E}_{\eta_{k}}
\left[f(x^k)-f^\ast\right]}{\epsilon}\leq\frac{1}{\epsilon}\left(1-\frac{2\mu_f}{(N-1)(1+\mu_f)}\right)^{\tilde{K}}
\tilde{R}^2(x^0)\\
&\leq&\frac{1}{\epsilon}\exp\left(-\frac{2\tilde{K}\mu_f}{(N-1)(1+\mu_f)}\right)
\tilde{R}^2(x^0)\leq\rho
\end{eqnarray*}
and hence the second claim holds.
\end{proof}

\section{Conclusions}
\label{sec:conclusions}
In this paper, we extended Nesterov's second technique \cite{Nesterov12} to analyze the convergence of the RBCD algorithm proposed in \cite{NecoaraNG15}    for solving  the large-scale optimization problem \eqref{eq:problem} with linearly coupled constraints.  Compared with those presented in \cite{NecoaraNG15}, our convergence rates are significantly sharper. We also derive better iteration complexities in probability for the method. In particular, when $f$ is a smooth convex function, our iteration complexity is smaller than the one implied in \cite{NecoaraNG15} by at least $N-1$. 

There are several directions for future work. Firstly, it remains an open question on how to extend the  techniques to analyze the convergence of RBCD methods for for composite optimization with linearly coupled constraints.  Secondly, it would be very interesting to generalize  different choices for the probability distribution for generating a single block for decoupled optimization problems \cite{Qu1,Qu2, Rich,Zhang}  to the case of generating pairs of blocks for the optimization problem \eqref{eq:problem}. 

\section*{Acknowledgments}
This work was done during Drs Fan and Xu were visiting the Department of Mathematics and Statistics, SUNY Albany. 
The work by Qin Fan and Min Xu was supported by the National Nature Science Foundation of China (11301052, 11301045) and the Fundamental Research Funds for the Central Universities (DUT16LK33, DUT15RC(3)058). The work by Yiming Ying described in this paper is supported by the Simons Foundation (\#422504) and the 2016-2017 Presidential Innovation Fund for Research and Scholarship (PIFRS) program from SUNY Albany.

\end{document}